\documentclass[a4paper,12pt]{article}

\usepackage{amsmath,amsthm,array}
\usepackage{amssymb}
\usepackage{mathtools}
\usepackage{enumerate}
\usepackage[shortlabels]{enumitem}
\usepackage{tikz-cd}
\usepackage[margin=2cm]{geometry}
 
\usepackage[utf8]{inputenc}
\usepackage{tgtermes}
\usepackage{mathptmx}
\usepackage[T1]{fontenc}
\usepackage{stackrel}
\usepackage{pict2e,picture}
\usepackage{xcolor}

\usepackage{comment}

\usepackage{hyperref}
\hypersetup{colorlinks=true}

\usepackage{fixme}
\fxsetup{status=draft, layout=inline, author=, theme=color}

\fxusetheme{color}

\definecolor{fxnote}{rgb}{1,0,0}

\newtheorem{theorem}{Theorem}[section]
\newtheorem{lemma}[theorem]{Lemma}
\newtheorem{proposition}[theorem]{Proposition}
\newtheorem{corollary}[theorem]{Corollary}
\newtheorem{definition}[theorem]{Definition}

\theoremstyle{definition}

\newtheorem{remark}[theorem]{Remark}

\newcommand{\ZZ}{\mathbb{Z}}
\newcommand{\NN}{\mathbb{N}}

\newcommand{\cP}{\mathcal{P}}
\newcommand{\cS}{\mathcal{S}}
\newcommand{\cU}{\mathcal{U}}
   
\newcommand{\cF}{\mathcal{F}}

\newcommand{\Jac}{\operatorname{Jac}}

\newcommand{\End}{\operatorname{End}}

\newcommand{\Specm}{\operatorname{Specm}}

\newcommand{\kar}{\operatorname{char}}

\newcommand{\op }{{\operatorname{op}}}

\newcommand{\uprod}{\prod_{\cU}}
\newcommand{\fprod}{\prod_{\cF}}

\newcommand{\leftbimod}[3]{\vphantom{#1}^{#2}{\kern-#3pt #1}}

\usepackage{fixme}
\fxsetup{status=draft, layout=inline, author=, theme=color}

\numberwithin{equation}{subsection}
  
\begin{document}

\title{Finitely generated bimodules over Weyl algebras}

\date{}

\author{Niels Lauritzen and Jesper Funch Thomsen} 

\maketitle 

\begin{abstract}
Let $A$ be the $n$-th Weyl algebra over a field of characteristic zero, and $\varphi:A\rightarrow A$ an endomorphism with $S = \varphi(A)$. We prove that if $A$ is finitely generated as a left or right $S$-module, then $S = A$. The proof involves reduction to large positive characteristics. By holonomicity, $A$ is always finitely generated as an $S$-bimodule. Moreover, if this bimodule property could be transferred into a similar property in large positive characteristics, then we could again conclude that $A=S$. The latter would imply the Dixmier Conjecture.
\end{abstract}

\section*{Introduction}

Let $\varphi:A\rightarrow A$ denote an endomorphism of the $n$-th Weyl algebra over a field $k$. Let $S=\varphi(A)$ denote the image of $\varphi$, and consider $A$ as a bimodule over $S$ by left and right multiplication. When $k$ is a field of characteristic zero, it is known \cite{Bavula} that $A$ is holonomic as an $S$-bimodule (see \cite{LTG} for an interpretation of this result using tensor products of bimodules), and therefore, $A$ is finitely generated (in fact, cyclic) as an $S$-bimodule. In positive characteristic, the Weyl algebra is an Azumaya algebra, and thus, finite generation of $A$ as an $S$-bimodule is actually equivalent to finite generation of $A$ as an $S$-module from either left or right. In large characteristics, this would imply that $\varphi$ induces a finite étale endomorphism of the center of $A$, which leads to the conclusion that $A=S$.

Therefore, it becomes interesting to transfer the finite generation of $A$ as an $S$-bimodule in characteristic zero to large positive characteristics. This would imply $A=S$ in characteristic zero, providing a proof of Dixmier's conjecture also resolving the Jacobian conjecture \cite[p.~297]{Bass}, \cite{Kontsevich}, \cite{Tsuchimoto1}, which is a long standing open problem \cite[p.~301]{Keller}. Regrettably, the direct deduction of finite generation of $A$ as an $S$-bimodule in large positive characteristics from its finite generation in characteristic zero remains elusive. Nevertheless, the condition of being finitely generated from either left or right translates nicely to large positive characteristics. Consequently, if $A$ is finitely generated as an $S$-module from either left or right, we can conclude $A=S$ in characteristic zero. This result answers the last question
posed in \cite{LT} affirmatively.

To make the positive characteristic reduction work, a suitable version of the simply connectedness of
affine space in positive characteristic is needed (see Theorem \ref{theorem:finetale})\footnote{Michel Brion has told us that a refinement of
the suitable positive characteristic version of simply connectedness follows from a computation 
of the prime to $p$ fundamental group of affine space using results 
by Orgogozo.
We will address this in \cite{LTQ}, where Weyl algebra
endomorphisms are treated exclusively in positive characteristics.}. 
This involves reduction to
characteristic zero and occupies Section \ref{section:polynomialring} in the last part of the paper.
Section \ref{section:polynomialring} does not depend on the previous sections and can be read
separately. Along the way, we give a very simple proof (see Theorem \ref{theorem:Gabber})
of the result \cite[(1.4) Corollary]{Bass} that if $\varphi$ is an automorphism of affine $n$-space, 
then $\deg(\varphi^{-1})\leq \deg(\varphi)^{n-1}$.

In the end of the paper we sketch how degree bounds for Gr\"obner bases also lead
to the simply connectedness in positive characteristic.

The most accessible way to handle transitions between characteristic zero and
positive characteristic in our setup is
through ultraproducts. We begin by briefly introducing this approach.

\section{Ultraproducts}

In this section $X$ denotes a non-empty set and $R$ a commutative ring.

\subsection{Filters}
A \emph{filter} (see also \cite[Chapter I, \S 6]{BourbakiTop1-4}) 
on $X$ is a subset $\cF\subset 2^X$ with 
$\emptyset\not\in \cF$, $A\cap B\in \cF$ if $A, B\in \cF$ 
and $B\in \cF$ if $A\in \cF$ and $B\supset A$, where $A$ and  $B$ are subsets of $X$.

A maximal filter is called an \emph{ultrafilter}. 
The set $\cF_a = \{S\subset X \mid a\in S\}$ of all subsets of $X$ containing a specific $a\in X$ is an ultrafilter on $X$. Such a filter is called \emph{principal}.
The properties below follow (almost) 
immediately from the definition of a filter.

\begin{enumerate}[label=(\roman*)]

\item 
Every filter is contained in
an ultrafilter (by Zorn's lemma).
\item 
If $\cS$ is a non-empty collection of subsets of $X$ and
 $S_1 \cap\cdots \cap S_n\neq \emptyset$ for
 finitely many $S_1, \dots, S_n\in \cS$, then
 $$
 \{Z\subset X \mid Z\supset S_1 \cap \cdots \cap S_n\  
 \text{for finitely many $S_1, \dots, S_n\in \cS$}\}
 $$
 is a filter containing $\cS$.
 \item 
An ultrafilter $\cU$ has
 the property that
 $A\in \cU\text{or } X\setminus A\in \cU$
 for every subset $A\subset X$.
\end{enumerate}
All ultrafilters on finite sets are principal. On
infinite sets there exists non-principal ultrafilters.
In fact,
\begin{enumerate}[resume*]
 \item 
 An ultrafilter contains the (filter of) cofinite subsets on an infinite set if and 
 only if it is non-principal.
 \end{enumerate}

\subsection{Ultraproducts}

Let $(S_i)_{i\in X}$ be a family of sets. A filter $\cF\neq \emptyset$ on
    $X$ defines an equivalence relation on $\prod_{i\in X} S_i$
    given by $(s_i)\sim (t_i)$ if and only if 
    $\{i \mid s_i = t_i\}\in \cF$. We use the notation
    $$
    \prod_{\cF} S_i :=  \left(\prod_{i\in X} S_i\right) \bigg/ \sim.
    $$
    If $S_i$ are commutative rings, then 
    $I_\cF = \{(s_i) \mid (s_i) \sim 0\}$ is an ideal in
    $\prod_{i\in X} S_i$ and
    $$
    \prod_{\cF} S_i = \prod_{i\in X} S_i \bigg/ I_\cF.
    $$
    Suppose that $S_i$ are fields. Then $\cF \mapsto I_\cF$ gives
    an inclusion preserving correspondence between filters on $X$
    and proper ideals in $\prod_{i\in X} S_i$. 
    The filter corresponding
    to a proper ideal $I$ is $\{Z(s) \mid s\in I\}$, where
    $Z(s) = \{i\in X \mid s_i = 0\}$. In particular, we get that 
    $\uprod S_i$ is a field if $\cU$ is an ultrafilter on $X$.

Let $\Specm(R)$ denote the set of 
maximal ideals of $R$ and $\Jac(R)$ the Jacobson radical of $R$ i.e., the intersection of all maximal ideals of $R$.

  \begin{proposition}\label{proposition:ultrafilterfield}
  Let $R$ be an integral domain with $\Jac(R) = \{0\}$,
  $X = \Specm(R)$ and $X_r = \{\mathfrak{m}\in X\mid r\notin \mathfrak{m}\}$
  for $r\in R$.
  Then
  $$
  \cF = \{Z\subseteq X \mid X_r \subseteq Z \text{ for some $r\in R\setminus\{0\}$}\}
  $$
  is a filter on $X$. Let $\cU$ be an ultrafilter on $X$ containing $\cF$.
  Then $\uprod R/\mathfrak{m}$ contains the fraction field of $R$.
\end{proposition}
\begin{proof}
  Notice that $X_r = \emptyset$ if and only if $r = 0$ by the assumption 
  $\Jac(R) = \{0\}$. Therefore $\emptyset\notin \cF$.
  Let $Z_1, Z_2\in \cF$ with
   $X_r \subseteq Z_1$ and $X_s \subseteq Z_2$, where $r, s\in R\setminus\{0\}$. 
   Then $r s\in R\setminus \{0\}$ and $X_{rs} \subseteq Z_1\cap Z_2$. Clearly, 
    $Z_2\in \cF$ if $Z_1\in \cF$ and $Z_1\subseteq Z_2$.

    Let 
    $$
    \pi: R\to \uprod R/\mathfrak{m}
    $$
    be the canonical map. Suppose that $r\in R\setminus \{0\}$ and 
    $\pi(r) = 0$. Then $r\in \mathfrak{m}$ for every 
    $\mathfrak{m}\in Z$ for some $Z\in \cU$ so that 
    $Z\subseteq X\setminus X_r$. Therefore $X\setminus X_r\in \cU$
    contradicting that $X_r\in \cU$. It follows that $\pi$ is injective and
    therefore that $\uprod R/\mathfrak{m}$ contains the fraction field of $R$.
\end{proof}

For $F\subset R[x_1, \dots, x_n]$, we let $V_S(F) = \{\alpha\in S^n \mid f(\alpha) = 0 \text{ for all } f\in F\}$,
where $S$ is an $R$-algebra.

\begin{lemma}\label{lemma:redmodp}
  Let  $\cF$ be a filter on 
  a set $X$. 
  Consider $F = \{f_1, \dots, f_m\} \subset 
  R[x_1, \dots, x_n]$  and let 
  $$
  S = \fprod S_j, 
  $$
  where 
  $\{S_j \mid j\in X\}$ is a set of $R$-algebras.
  Fix elements $\gamma_i = (\gamma_{i,j})_{j\in X}$
  in $\prod_{j\in X} S_j$, for $i=1,2,\dots,n$. Then 
  $([\gamma_1], [\gamma_2],
  \dots, [\gamma_n])\in V_S(F)$
  if and only if there exists $B \in \cF$ such 
  that $(\gamma_{1,j}, \gamma_{2,j}, \dots, \gamma_{n,j})\in V_{S_j}(F)$ 
  for $j \in B$. 
\end{lemma}
\begin{proof}
Notice that if $f \in R[x_1,\dots,x_n]$ is a 
polynomial, then we have the following identity in $S$
\begin{equation}\label{eq:crux}
 f([\gamma_1], [\gamma_2],\dots, [\gamma_n]) = [\big( f(\gamma_{1,j}, \gamma_{2,j},
   \dots, \gamma_{n,j})\big)_{j \in X}].
   \end{equation}
In particular, $f([\gamma_1], [\gamma_2],
  \dots, [\gamma_n])=0$ if and only if
  there exists a $B \in \cF$, such that
  $f(\gamma_{1,j}, \gamma_{2,j}, \dots, \gamma_{n,j})=0$ 
  for $j \in B$. Using that $\cF$ is 
  stable under finite intersections, the claim is a
  direct consequence of \eqref{eq:crux}.
\end{proof}

\section{The Weyl algebra}

In the following $\NN = \{0, 1, 2\, \dots\}$ denotes the natural numbers. For $n\in \NN$ and an $n$-tuple
 $a = (a_1, \dots, a_n)$ of elements in a ring $A$, we use the notation $a^v$ 
 with $v = (v_1, \dots, v_n)\in \ZZ^n$ to denote
$$
a_1^{v_1}\cdots a_n^{v_n}
$$
if $v\in \NN^n$ and $0$ if
$v\in \ZZ^n \setminus \NN^n$. 
We will use $a\in \ZZ$ to denote $(a, \dots, a)\in \ZZ^n$, when it fits the
context. For integral vectors $u = (u_1, \dots, u_n)$ and $v = (v_1, \dots, v_n)$, we let $u\leq v$ denote
the partial order given by $u_1\leq v_1, \dots, u_n \leq v_n$.
For $a, b\in A$, we let $[a, b] = ab - ba$.

Let $R$ denote a commutative ring.
For $n\in \NN$, $A_n(R)$ denotes the $n$-th Weyl algebra 
over $R$. This is the (free) $R$-algebra generated by $2n$
variables $x = (x_1, \dots, x_n), \partial=(\partial_1, \dots, \partial_n)$
with relations
\begin{align}
  \begin{split}\label{walgrels}
    [x_i, x_j] &= 0\\
    [\partial_i, \partial_j] &= 0\\
    [\partial_i, x_j] &= \delta_{ij}
  \end{split}
\end{align}
for $i, j = 1, \dots, n$. 
  Elements in
    $$
    S_n = \{x^u \partial^v \mid u, v\in \NN^n\}\subset A_n(R)
    $$
    are called standard monomials. We have not been 
    able to track down a proof of the following result (for
    general commutative rings). We therefore outline a
    proof (of linear independence) shown to us by J.~C.~Jantzen.

\begin{proposition}\label{proposition:stdbasis}
  $A_n(R)$ is a free $R$-module with basis $S_n$.
\end{proposition}
\begin{proof}
    Using the relations \eqref{walgrels} it follows
    that $S_n$ spans $A_n(R)$ as an $R$-module. To prove
    linear independence of $S_n$, we let $A_n(R)$ act on the
    free $R$-module with basis
    $$
    M = \{X^u D^v \mid u, v\in \NN^n\}
    $$
    through
    \begin{align*}
        x_i\cdot X^u D^v &= X^{u + e_i} D^v\\
        \partial_i\cdot X^u D^v &= X^u D^{v + e_i} + u_i X^{u-e_i} D^v
    \end{align*}
    for $i = 1, \dots, n$, where $e_i$ denotes the $i$-th canonical basis vector.
    Suppose that 
    $$
    \Delta = \sum_{u, v\in \NN^n} c_{uv} x^u \partial^v = 0
    $$
    in $A_n(R)$.
    Then acting on $M$ we get
    $$
    \Delta\cdot X^0 D^0 = \sum_{u, v\in \NN^n} c_{uv} X^u D^v = 0
    $$
    showing that $c_{uv}=0$ thereby proving the $R$-linear independence of $S_n$.
\end{proof}

\begin{definition}\label{definition:stddegree}
  The degree of $f\in A_n(R)\setminus\{0\}$, where
  $$
  f = \sum_{u, v\in \NN^n} a_{uv} x^u \partial^v
  $$
  is defined as
  $$
  \deg(f) = \max\{|u| + |v| \mid a_{uv}\neq 0\},
  $$
  where 
  $$
|w| = w_1 + \dots + w_n
$$
for $w = (w_1, \dots, w_n)\in \NN^n$. 
The
degree of an endomorphism 
$\varphi:A_n(R)\rightarrow A_n(R)$ is defined as 
$$
\deg(\varphi) = \max\{\deg(\varphi(x_1)), \dots, 
\deg(\varphi(x_n)), \deg(\varphi(\partial_1)), \dots, 
\deg(\varphi(\partial_n))\}.
$$
\end{definition}

\subsection{The Weyl algebra in positive characteristic}

\begin{theorem}\label{theorem:JAlg}
  Suppose that $\kar(R) = p$, where $p$ is a prime number and let
  $$
  C = R[x_1^p, \dots, x_n^p, \partial_1^p, \dots, \partial_n^p]\subset A:=A_n(R).
  $$  
  Then 
  \begin{enumerate}[(i)]
  \item \label{item:centerofAn}
  The center of $A$ is equal to $C$.
  \item \label{item:pexpansion}
  If $Q_1, \dots, Q_n, P_1, \dots, P_n\in A$ satisfy the
  commutation relations for the Weyl algebra i.e.,
  \begin{align*}
  \begin{split} 
  [Q_i, Q_j] &= 0\\
  [P_i, P_j] &= 0\\
  [P_i, Q_j] &= \delta_{ij}
  \end{split}
  \end{align*}
  for $i, j = 1, \dots, n$, then 
  $$
  \{Q^\alpha P^\beta \mid \alpha, \beta \in \NN^n, 0 \leq \alpha, \beta \leq p-1\}
  $$ 
   is a basis for $A$ as a module over $C$.
  \item \label{Abimod}
  $A$ is an Azumaya algebra over $C$ and (therefore) 
  $F(N) = N\otimes_C A$ defines an equivalence between the category of $C$-modules and 
  the category of $C$-linear $A$-bimodules (i.e., $A$-bimodules $M$ satisfying $r m = m r,$ for $r\in C$ and $m\in M$) with inverse 
  $$
  G(M) = M^A = \{m\in M \mid a m = m a, \forall a\in A\}.
  $$
  Here $N$ denotes a $C$-module and $M$ a $C$-linear $A$-bimodule  . The
  equivalence preserves finitely generated modules.
\end{enumerate}
\end{theorem}
\begin{proof}
    See \cite[Theorem 1.7]{LT} for a proof of \ref{item:centerofAn} and 
    \ref{item:pexpansion}. 
    Let $A^e = A\otimes_C A^{\op}.$
    To prove \ref{Abimod}, we need to show (cf. \cite[Theorem III.5.1, 2)]{KnusOjanguren}) that the natural map
    \begin{equation}\label{eq:Azumaya}
    \varphi: A^e\rightarrow \End_C(A)
    \end{equation}
    of $C$-algebras given by $\varphi(a\otimes b)(x) = a x b$
    is an isomorphism. Notice that
    $A^e$ and $\End_C(A)$ are free $C$-modules of rank $p^{4n}$. The elements
    $\partial_1\otimes 1, \dots, \partial_n\otimes 1, 1\otimes \partial_1, \dots, 1\otimes \partial_n$ and
    $x_1\otimes 1, \dots, x_n\otimes 1, 1\otimes x_1, \dots, 1\otimes x_n$
    define $a_1, \dots, a_{2n}, b_1, \dots, b_{2n}\in A^e$ with 
    $[a_i, a_j] = [b_i, b_j] = 0$ and $[a_i, b_j] = \delta_{ij}$ for $i, j = 1, \dots, 2n$.
     
    As in the proof of Theorem 1.7$(ii)$ in \cite{LT}, it follows that 
    $$
    M = \{\alpha^u \beta^v \mid u, v\in \NN^{2n}, 0\leq u, v\leq p-1\}\subset \End_C(A)
    $$
    is a $C$-basis of $\End_C(A)$, where 
    $\alpha = (\varphi(a_1), \dots, \varphi(a_{2n}))$ and $\beta = (\varphi(b_1), \dots, \varphi(b_{2n}))$.
    Therefore $\varphi$ is an isomorphism of $C$-modules and $A$ is an Azumaya algebra over $C$. This
    implies that
    $F$ and $G$ is an adjoint pair of inverse equivalences between the category of $C$-modules and
    the category of $C$-linear $A$-bimodules by \cite[Theorem III.5.1, 3)]{KnusOjanguren}. 
    Finally, finite generation is a categorical property preserved under equivalences by \cite[Proposition 21.8]{AF}.  
 \end{proof}
  
\subsection{Endomorphisms of Weyl algebras}

For an arbitrary ring $S$ and a ring endomorphism $\varphi:S \rightarrow S$, we let
${}_\varphi S_\varphi$ denote the $S$-bimodule $S$ with  multiplication
$$
x s y = \varphi(x) s \varphi(y),
$$
where $s, x, y\in S$. Similarly ${}_\varphi S$ denotes $S$ as a left module with multiplication given by $x s = \varphi(x) s$. 
Using the notation from Theorem \ref{theorem:JAlg}, we have the following corollary.

\begin{corollary}\label{corollary:WeylAlgp}
Let $\varphi:A \rightarrow A$ be an $R$-algebra homomorphism. Then
\begin{enumerate}[(i)]
\item\label{item:varphiR} 
$\varphi(C)\subset C$
  \item\label{item:varphiiso}
  $\varphi$
   is an isomorphism if and only if  $\varphi|_C$ is an isomorphism. 
  \item\label{item:fingen}
   The bimodule ${}_\varphi A_\varphi$ is $C$-linear with $({}_\varphi A_\varphi)^A ={}_\varphi C$.
   Furthermore, 
  ${}_\varphi A_\varphi$ is finitely generated if and only if $\varphi|_C$ is finite.
\end{enumerate}
  \end{corollary}
\begin{proof}
    The first two claims are consequences of Theorem \ref{theorem:JAlg}
    $\ref{item:pexpansion}$
    with $P_i = \varphi(\partial_i)$ and $Q_i = \varphi(x_i)$ for
    $i = 1, \dots, n$. For the proof of $\ref{item:fingen}$, we note that the 
    $C$-linearity of ${}_\varphi A_\varphi$ follows from $\ref{item:varphiR}$ and 
  that $({}_\varphi A_\varphi)^A = {}_\varphi C$ follows from Theorem 
  \ref{theorem:JAlg}$\ref{item:pexpansion}$. The last statement in $\ref{item:fingen}$
  follows from Theorem \ref{theorem:JAlg}$\ref{Abimod}$.
\end{proof}

For completeness we include the proof of the following well known result (see \cite[Corollary 3.4]{Bavula1} for
the characteristic zero case).

\begin{theorem}\label{theorem:invbound}
  Let $R$ be a reduced commutative ring.
    If $\varphi:A_n(R)\rightarrow A_n(R)$ is an
    automorphism of $R$-algebras, then
    $$
    \deg(\varphi^{-1}) \leq \deg(\varphi)^{2n-1}.
    $$
\end{theorem} 
\begin{proof} 
  Using Proposition \ref{proposition:stdbasis} and expanding 
  $\varphi(x_i), \varphi(\partial_j), \varphi^{-1}(x_i), \varphi^{-1}(\partial_j)$ in the basis of the
  standard monomials for $i, j = 1, \dots, n$,
  we may assume that $R$ is finitely generated over $\ZZ$ and therefore
  a Jacobson ring (see \cite[Chapter V, \S3.4]{BourbakiCA}).
  Since $R$ is reduced, $\Jac(R) = 0$ and there exists a maximal ideal $\mathfrak{m}$ 
  not containing the coefficient of a 
  highest degree standard monomial in the definition of $\deg(\varphi^{-1})$ 
  (cf. Definition \ref{definition:stddegree}).
  But $k = R/\mathfrak{m}$ is a field of characteristic $p>0$ and $\varphi, 
  \varphi^{-1}$ induce inverse automorphisms $\bar{\varphi}, \bar{\varphi}^{-1}$ 
  of $A_n(k)$ with $\deg(\bar{\varphi}^{-1}) = \deg(\varphi^{-1})$ and
  $\deg(\bar{\varphi}) \leq  \deg(\varphi)$.  

  Now the result follows by restricting $\bar{\varphi}$ to the center $C$ of $A_n(k)$ 
  (cf. Corollary \ref{corollary:WeylAlgp}\ref{item:varphiiso}) and using
  Theorem \ref{theorem:Gabber} as $\deg(\bar{\varphi}|_C) = \deg(\bar{\varphi})$. 
\end{proof}

\begin{remark}
  Notice that the bound in Theorem \ref{theorem:invbound} only depends on
  $\deg(\varphi)$ and $n$. In particular, it does not depend on the (reduced) ring $R$.
  We can only prove Theorem \ref{theorem:invbound} when
  $R$ is reduced. The independence of $R$ in bounding the degree of the inverse
  in Theorem \ref{theorem:invbound} for nilpotent rings should be equivalent to 
  the Dixmier conjecture in analogy with \cite[Theorem (1.1)]{Bassinvdeg} and 
  \cite[Ch. I, Prop. (1.2)]{Bass}.
\end{remark}  

The following result is non-trivial and  crucial.

\begin{lemma}\label{lemma:etalerestrict}
  Let $k$ be a field of characteristic $p>0$ and $\varphi$ an endomorphism of $A_n(k)$. 
  Then $\varphi$ restricts to an \'etale endomorphism of the center of $A_n(k)$ if $p > 2 \deg(\varphi).$
\end{lemma}
\begin{proof}
This is a consequence of \cite[Corollary 3.3]{Tsuchimoto1}.
\end{proof} 

\begin{remark}
  In \cite{LTQ}, the bound 
  in Lemma \ref{lemma:etalerestrict} is strengthened to $p > \deg(\varphi)$.  
\end{remark}

\section{The main result}

The following lemma is the positive characteristic version of Theorem \ref{theorem:main}, which is our main result. It basically
uses the equivalence between bimodules and modules over the center (cf. Theorem \ref{theorem:JAlg}$\ref{Abimod}$) to reduce to
the commutative case of finite \'etale endomorphisms of polynomial rings, which is
the focus of section \ref{section:polynomialring} and in particular Theorem \ref{theorem:finetale}. 
Notice again that it is crucial that 
an endomorphism of the Weyl algebra restricts to an \'etale endomorphism of the center
if the characteristic is large enough (cf. Lemma \ref{lemma:etalerestrict}).

\begin{lemma}\label{lemma:main}
  Let $k$ be a field of characteristic $p>0$, $A = A_n(k)$ and 
  $\varphi: A\rightarrow A$ an endomorphism of degree $\leq d$. 
  Then there exists a uniform bound $D = D(n, d)$ only depending on $n$ and $d$, such that if $p > D$ and 
  ${}_\varphi A_\varphi$ is finitely generated, then $\varphi$ is an isomorphism.
\end{lemma}
\begin{proof}
Let $C$ be the center of $A_n(k)$.
    By Lemma \ref{lemma:etalerestrict}, 
    $\varphi|_C$ is \'etale for $p > 2 d$. The assumption that ${}_\varphi A_\varphi$ is
    finitely generated shows that
    $\varphi|_C$ is finite by Corollary \ref{corollary:WeylAlgp}$\ref{item:fingen}$. 
    Let $D(n, d) = \max\{2 d, N(2 n, d)\}$, where $N(2 n , d)$ refers to the
    uniform bound in Theorem \ref{theorem:finetale}.
    Then $\varphi|_C$ is an
    isomorphism for $p > D(n, d)$ by Theorem \ref{theorem:finetale}
    and finally we conclude that $\varphi$ is an isomorphism
    by Corollary \ref{corollary:WeylAlgp}$\ref{item:varphiiso}$.
\end{proof}

To go to characteristic zero we need the following characterization of (left) generating sets.

\begin{lemma}\label{SAtoAS}
  Let $R$ be a commutative ring, $A = A_n(R)$, $\varphi: A\rightarrow A$ 
  an $R$-algebra endomorphism and $S = \varphi(A)$.
Suppose that $G = \{g_1, \dots, g_m\}\subset A$ satisfies
\begin{enumerate}[(i)]
\item $1\in G$
\item There exists $s_{ijl}, t_{ijl}\in S$ , such that
\begin{align}
g_j x_i &= s_{ij1} g_1 + \cdots + s_{ijm} g_m\label{xvar}\\
g_j \partial_i &= t_{ij1} g_1 + \cdots + t_{ijm} g_m\label{dvar}
\end{align}
for $l = 1, \dots, m$ and $1\leq i, j\leq n$. 
\end{enumerate}

Then
$G$ generates $A$ as a left $S$-module i.e.,
$$
A = S g_1 \cdots + S g_m.
$$  
Conversely if $G = \{g_1, \dots, g_m\}$ is a generating set for $A$ as a left $S$-module
with $1\in G$, then $G$ satisfies \eqref{xvar} and \eqref{dvar} above.
\end{lemma}
\begin{proof}
Let $M = S g_1 + \cdots + S g_m$. We wish to prove that
$M = A$.
Using \eqref{xvar} we get that
$
x^u\in M
$
for every $u\in \NN^n$. Building on this, one shows using \eqref{dvar} that
$
x^u \partial^v\in M
$
for every $u, v\in \NN^n$. Therefore we must have $M = A$. The converse follows
by definition of a (left) generating set.
\end{proof}

\begin{theorem}\label{theorem:main}
  Let $K$ be a field of characteristic $0$, $A = A_n(K)$ and 
  $\varphi: A\rightarrow A$ a $K$-algebra endomorphism with $S = \varphi(A)$. If $A$ is finitely generated
as a left or right module over $S$, then $S = A$.
\end{theorem}

\begin{proof}
  Assume that $A$ is generated by $G = \{g_1, \dots, g_m\}$ as a left $S$-module (the case of a right 
  $S$-module is similar) with $1\in G$.
  We start by fixing a finitely generated $\ZZ$-algebra $R'$ of $K$ generated by the 
  coefficients of $\varphi(x_i), \varphi(\partial_i)$ for $i = 1, \dots, n$ in the basis of standard
  monomials (cf.~Definition \ref{definition:stddegree}). In the same way we extend $R'$ to a finitely generated
  $\ZZ$-algebra $R$, such that the coefficients of $g_1, \dots, g_m$ and $s_{ijk}, t_{ijk}$ of Lemma \ref{SAtoAS} 
  are in $R$. For $\mathfrak{m}\in \Specm(R)$, $R/\mathfrak{m}$ is a field of characteristic $p>0$ and
  it follows by reduction modulo $\mathfrak{m}$ of \eqref{xvar} and \eqref{dvar} 
  that ${}_{\bar{\varphi}}A_n(R/\mathfrak{m})_{\bar{\varphi}}$ 
  is finitely generated as a left module and therefore as a bimodule, where 
  $\bar{\varphi}: A_n(R/\mathfrak{m})\rightarrow A_n(R/\mathfrak{m})$ denotes the reduction of $\varphi$ 
  modulo $\mathfrak{m}$.

  Now let $D = D(n, \deg(\varphi))$ be the bound from Lemma \ref{lemma:main}. 
  Let $M$ be the product of the prime numbers $\leq D$. Fix an ultrafilter $\cU$
  on $\Specm(R)$
  as in Proposition \ref{proposition:ultrafilterfield}. Then 
  $\bar{\varphi}$ is an isomorphism for every $\mathfrak{m}\in X_M$ (with the notation in Proposition 
  \ref{proposition:ultrafilterfield}) by Lemma \ref{lemma:main} and 
  Theorem \ref{theorem:invbound} along with Lemma \ref{lemma:redmodp} show that $\varphi$ 
  is an isomorphism over the 
  ultraproduct $\prod_\cU R/\mathfrak{m}$. Therefore $\varphi$ is an isomorphism over $K$ by
  \cite[Lemma 3.3]{LT}.
\end{proof} 

The last part of the paper concerns the proof of Theorem \ref{theorem:finetale} below, 
which is the key result needed in the proof of Lemma \ref{lemma:main} (and therefore Theorem \ref{theorem:main})
above.

\section{Bounds on polynomial endomorphisms}\label{section:polynomialring}

In this section we deal exclusively with commutative rings and introduce new notation independent of the notation in the
previous sections i.e., $A$ no longer refers to the Weyl algebra.

Let $B = K[x_1, \dots, x_n]$ be the polynomial ring 
of $n$ variables over an arbitrary field $K$ (with $n>0$), 
$\varphi:B\rightarrow B$ a $K$-algebra homomorphism
given by $\varphi(x_i) = f_i$ and $\varphi(B)=:A\subset B$. Denote
the fields of fractions of $A$ and $B$ by $K(A)$ and $K(B)$ respectively.
We let 
\begin{equation}\label{eq:degcommutativephi}
\deg(\varphi) = \max\{\deg(f_1), \dots, \deg(f_n)\}.
\end{equation}
where  $\deg(f_i)$ denotes the total degree of 
$f_i$. For $d\in \NN$ we let 
$B_{\leq d}$ denote the $K$-subspace of $B$ consisting of polynomials 
$f$ with $\deg(f) \leq d$. 

\begin{lemma}
\label{lemmabounds}
Let $d, r\in \NN$, with $r>0$, and 
$h_1,h_2,\dots,h_r \in B_{\leq d}$. Suppose that $m> n+d$
is an integer and that $h_1, \dots, h_r$
do not satisfy a non-trivial linear relation
$$
\varphi(b_1) h_1 + \cdots + \varphi(b_r) h_r = 0
$$
with $b_i\in B_{\leq m}$ for $i = 1, \dots, r$.
Then
$$ m\, (r-\deg(\varphi)^n)  <   2^n \deg(\varphi)^{n-1} (n+d).$$  
\end{lemma}
\begin{proof}
Notice first that $\deg(\varphi)>0$ with the given assumptions. 
Let $V$ denote the subspace 
$$
V = \left\{ \varphi(b_1) h_1 + \cdots + \varphi(b_r) h_r\,\, \bigg|\,\, b_1, \dots, b_r\in B_{\leq m} \right\} \subseteq B_{\leq m \deg(\varphi)+d}.
$$
By 
assumption
$$ \dim_K(V) = r \, \dim_K(B_{\leq m})\leq \dim_K(B_{\leq m \cdot 
\deg(\varphi) + d}).$$
Therefore
$$ r \frac{m^n}{n!} \leq  r \binom{m+n}{n} \leq \binom{m \, \deg(\varphi)+d+n}{n} \leq \frac{(m \, \deg(\varphi)+d+n)^n}{n!}$$
showing that
\begin{equation}\label{ineq1}
      r m^n  \leq (m \, \deg(\varphi)+d+n)^n < (m \, \deg(\varphi))^n + 2^n (m \, \deg(\varphi))^{n-1} (d+n),
\end{equation}
where the last inequality follows by the assumption $d+n< m$ and as 
$\deg(\varphi)>0$. The claimed inequality now follows from \eqref{ineq1}. 
\end{proof}

In the following we will assume that $f_1, f_2, \dots, f_n$ are 
algebraically independent over $K$. The induced field extension
 $K(A) \subseteq K(B)$  is then
finite.

\begin{proposition} 
\label{proposition1}
$$[K(B):K(A)]\leq \deg(\varphi)^n$$
\end{proposition}
\begin{proof}
Let $h_1, h_2, \dots, h_r$ denote a basis for 
$K(B)$ as a vector space over 
$K(A)$. Clearing denominators 
we may assume that each $h_i$ is contained in 
$B$. Choose $d$ such that $h_i \in B_{\leq d}$, for each $i$.
Applying Lemma \ref{lemmabounds} we then conclude
that
$$(r-\deg(\varphi)^n) \cdot m <   2^n \deg(\varphi)^{n-1} (n+d),$$
for every positive integer $m>n+d$. Therefore
$r-\deg(\varphi)^n\leq 0$ and $r\leq \deg(\varphi)^n$.
\end{proof}

\subsection{The classical bound on the degree of the inverse}

We will give a short and elementary proof of the following classical result inspired by
ideas in \cite{Yu}.

\begin{theorem}[\cite{Bass}, Thm. 1.5]\label{theorem:Gabber}
If $\varphi$ is an automorphism of $B$, then $\deg(\varphi^{-1}) 
\leq \deg(\varphi)^{n-1}$.
\end{theorem}
\begin{proof}
Suppose that $\varphi^{-1}(x_i) = g_i$ for $i=1, \dots, n$. 
We may assume that $K$ is an infinite field.
Thus by
a linear change of coordinates we can assume, that each
$g_i$ is a monic polynomial in $x_n$ with coefficients 
in the subring $B'=K[x_1,x_2,\dots,x_{n-1}]$ of $B$.
Let $N$ denote the degree of $\varphi^{-1}$ and assume that $\deg(g_n)=N$. 
Now identify $B'$ with the 
quotient ring $B/(x_n)$ and use the notation $\overline{h}$ for the
image of $h\in B$ in $B/(x_n)$. 

Since $\varphi$ is an automorphism,
$B=K[f_1,f_2,\dots,f_n]$ and $B'=K[\overline{f_1},
\overline{f_2}, \dots, \overline{f_{n}}]$. We claim 
that  $\overline{f_1},\overline{f_2}, \dots, \overline{f_{n-1}}$
are algebraically independent over $K$ and that 
$$[K(x_1,x_2,\dots,x_{n-1}) : K(\overline{f_1},\overline{f_2}, \dots, \overline{f_{n-1}})] = \deg(\varphi^{-1}).$$
By Proposition \ref{proposition1} this will end the proof, as 
$\deg(\overline{f_i}) \leq \deg(\varphi)$. Write 
$$ g_n = a_0 + a_1 x_n + a_2 x_n^2 + \cdots + a_{N-1} x_n^{N-1}
+ x_n^N$$  
with $a_i \in B'$. Then 
$$ x_n = \varphi(g_n) = \varphi(a_0) + \varphi(a_1) f_n + 
\cdots + \varphi(a_{N-1}) f_n^{N-1} + f_n^N,$$
and 
\begin{equation}
\label{minimalfn}
c_0  + c_1\cdot \overline{f_n}
+ \cdots +  c_{N-1} \cdot \overline{f_n}^{N-1}
+ \overline{f_n}^N = 0.
\end{equation}
where $c_i = \overline{\varphi(a_i)}$ are elements in $A'=K[\overline{f_1},\overline{f_2}, \dots, \overline{f_{n-1}}]$.
It follows that $\overline{f_n}$ is integral over $A'$ and therefore that $B'=K[\overline{f_1},\overline{f_2}, \dots, 
\overline{f_{n}}]$ is integral over $A'$. We conclude, that the elements $\overline{f_1},\overline{f_2}, \dots, \overline{f_{n-1}}$ are algebraically independent and that 
$[K(x_1,x_2,\dots,x_{n-1}) : K(\overline{f_1},\overline{f_2}, \dots, \overline{f_{n-1}})] $
equals the degree of the minimal polynomial of $\overline{f_n}$
over $K(\overline{f_1},\overline{f_2}, \dots, \overline{f_{n-1}})$. 
Consider the polynomial 
$$ F=c_0  + c_1\cdot X
+ \cdots +  c_{N-1} \cdot X^{N-1}
+ X^N  \in A'[X].$$
By equation (\ref{minimalfn}) we see that $\overline{f_n}$ is
a root of $F$. So it suffices to prove that $F$ is irreducible
in $A'[X]$ (and thus also in $K(\overline{f_1},\overline{f_2}, 
\dots, \overline{f_{n-1}})[X]$). To see this  
observe first that $F$ is the image of
$g_n$ under the $K$-algebra automorphism 
$B \rightarrow A'[X]$ which maps $x_n$ to $X$ and 
$x_i$ to $\overline{f_i}$, for $i<n$. Then use that 
$g_n$ is irreducible as $\varphi$ is an automorphism. 
\end{proof}

\subsection{Integral extensions}

In the statement below we still work under the assumption 
that $f_1, f_2, \dots, f_n$ are algebraically independent.

\begin{proposition}
\label{proposition2}
Let $g$ denote an element in $B$ of degree $D$ which is integral 
over $A$, and let 
$$  
G = a_0 + a_1 T + a_2 T^2 + \cdots + a_{r-1} T^{r-1} + T^r
\in K(A)[T],$$
denote its minimal polynomial. Then 
$a_i\in \varphi(B_{\leq m})$ for $i=0, \dots, r-1$, where
$$
m = 2^n \deg(\varphi)^{n-1} (n + D \cdot \deg(\varphi)^n).
$$
\end{proposition}
\begin{proof}

Applying Lemma \ref{lemmabounds} with $r = \deg(\varphi)^n+1, d = D \deg(\varphi)^n$ and $m$ as above, we conclude that there exists 
$b_i \in B_{\leq m}$, for $i=0,1,\dots,\deg(\varphi)^n$ with
$$
\varphi(b_0)  + \varphi(b_1) g + \cdots + \varphi(b_{\deg(\varphi)^n}) g^{\deg(\varphi)^n} = 0,
$$
such that not all $b_i$ are zero. The polynomial 
$$ 
F= \varphi(b_0)   + \varphi(b_1) T + \cdots + \varphi(b_{\deg(\varphi)^n}) T^{\deg(\varphi)^n} \in A[T]$$
is thus divisible by $G$ i.e., $F = G \cdot H $ for some polynomial 
$H$ in $K(A)[T]$. By \cite[Proposition 5.15]{atiyah1969introduction}, $G\in A[T]$. Therefore
we may even conclude that $H \in A[T]$. Consider now $F, G$ and $H$
as polynomials in $f_1,f_2,\dots f_n$ with coefficients in $K[T]$.
Then the (total) degree of $F$ is $\leq m$ which implies, that the
same is true for its divisor $G$. This ends the proof. 
\end{proof}

Recall that the ring homomorphism $\varphi: B \rightarrow B$ is
called integral if $B$ is integral over the image $A=\varphi(B)$.
In the present setup $\varphi$ can only by integral if $f_1, f_2,
\dots, f_n$ are algebraically independent. Combining this observation
with Proposition \ref{proposition1} and Proposition \ref{proposition2}
above we find:

\begin{theorem}\label{theorem:finiteextension}
Let $m = 2^n \deg(\varphi)^{n-1} (n+\deg(\varphi)^n)$. The
following statements are equivalent.
\begin{enumerate}[(i)] 
    \item The ring homomorphism   $\varphi: B \rightarrow B$ 
    is integral.
    \item There exists nonzero monic polynomials $F_i \in A[T]$,
for $i=1,2,\dots,n$, of degree $\deg(\varphi)^n$ and with 
coefficients in $\varphi(B_{\leq m})$, such that $F_i(x_i)=0$.
\end{enumerate}
 \begin{proof}
The second statement clearly implies the first statement. So
assume that $A \subseteq B$ is an integral extension. By 
Proposition \ref{proposition2} we know that the minimal polynomial
$G_i$ for $x_i$ has coefficients in $\varphi(B_{\leq m})$ for each $i$. 
Moreover, by Proposition \ref{proposition1} these minimal 
polynomials have degree $\leq \deg(\varphi)^n$. Multiplying 
$G_i$ by a suitable power of $T$ we obtain a polynomial 
$F_i$ of degree $\deg(\varphi)^n$ which satisfies the statement
in (ii).
 \end{proof}
\end{theorem}

The result below is a key component in the proof of Lemma \ref{lemma:main} used in the proof of 
 Theorem \ref{theorem:main}.

\begin{theorem}\label{theorem:finetale}
Let $n, D\in \NN$ be positive. There exists a uniform bound $N(n,D)\in \NN$ only depending on $n$ and $D$, such that 
when $k$ is a field of characteristic $p>N(n,D)$ and $\varphi$ is a finite and \'etale $k$-algebra endomorphism of the polynomial ring $k[x_1, \dots, x_n]$
of degree $\leq D$, then $\varphi$ is an automorphism. 
 \end{theorem}
\begin{proof}
Suppose that the theorem is wrong. Then we may find an infinite sequence of 
primes $p_1 < p_2 < \dots$  and corresponding fields $k_1, k_2, \dots$ of
characteristic $\kar(k_i)=p_i$ together with finite and \'etale $k_i$-algebra 
endomorphisms $\varphi_i$ of $k_i[x_1, \dots, x_n]$ of degree $\leq D$,
which are not automorphisms. For simplicity we may assume that the 
determinant of the Jacobian matrix of each $\varphi_i$ equals $1$.
 
Choose a non-principal ultrafilter $\cU$ on the set $\cP= \{p_1, p_2, \dots\}$
and let $C$ denote the ultraproduct   $\uprod k_i$. Let further   
$$
\varphi: C[x_1, \dots, x_n] \rightarrow C[x_1, \dots, x_n]
$$
denote the $C$-algebra endomorphism induced by $\varphi_i, i >0$. 
We claim that $\varphi$ is a finite and \'etale. That $\varphi$
is \'etale follows as the determinant of the Jacobian matrix of 
each $\varphi_i$, and hence also of $\varphi$, equals $1$. To
conclude that $\varphi$ is finite we apply Theorem 
\ref{theorem:finiteextension}. It suffices to prove that 
each $x_j$ is integral over the image of $\varphi$. By 
assumption $x_j$ is integral over the image of $\varphi_i$,
so by Theorem \ref{theorem:finiteextension} we may find 
polynomials $p_{i,r} \in k_i[x_1,\dots,x_n]$, for 
$r=1,2,\dots,D^n$, of degree $\leq m := 2^n D^{n-1}(n+D^n)$,
such that $x_j$ is a root of the monic polynomial
$$ P_i(T) =  T^{D^n} + \sum_{r=1}^{D^n} \varphi_i(p_{i,r}) T^{r-1}. $$
Let now $p_r \in C[x_1,\dots,x_n]$ denote the polynomial defined 
by the collection of polynomials $(p_{i,r})_{i >0}$. Then $x_j$ is a root
of the monic polynomial
$$ P(T) =  T^{D^n} + \sum_{r=1}^{D^n} \varphi(p_{r}) T^{r-1}, $$
and we conclude that $\varphi$ is finite.

Next observe that $C$ is a field of characteristic zero and thus 
$\varphi$ is an automomorphism by \cite[(2.1) Theorem, (d)]{Bass}
with inverse $\varphi^{-1}$ of some degree $D'$. Fix, for each $i>0$, 
a lift $\psi_i$ of $\varphi^{-1}$ to a $k_i$-algebra endomorphism of $k_i[x_1,\dots,x_n]$ of degree $\leq D'$. 
In this way $\varphi^{-1}$ is defined from $\psi_i$, for $i>0$, in 
the same way as $\varphi$ was defined from $\varphi_i$, for $i>0$.
That $\psi_i$ is an inverse to $\varphi_i$ is then a polynomial  
condition in the coefficients of $\psi_i$ and $\varphi_i$. These 
polynomial conditions are satisfied for $\varphi^{-1}$ and $\varphi$, and 
thus $\psi_i$ and $\varphi_i$ will be mutually inverses for every
prime $p_i$ in some element $B$ in $\cU$ by Lemma \ref{lemma:redmodp}. This is a contradiction
as none of the endomorphisms $\varphi_i$, $i>0$, are automorphisms.
\end{proof}

\subsection{A Gr\"obner basis approach}

Here we outline an alternative approach to obtaining bounds similar to the ones in 
Theorem \ref{theorem:finiteextension}. We keep the notation of the
previous section.

It suffices to prove that if $f_1, \dots, f_n\in K[x_1, \dots, x_n]$
are polynomials of degree $\leq d$, then
$x_i$ is integral over $K[f_1, \dots, f_n]$ for $i = 1, \dots, n$ if and only
if the minimal polynomial for each $x_i$ has the form
\begin{equation}\label{eq:GBmonpols}
T^m + a_{m-1}(f_1, \dots, f_n) T^{m-1} + \cdots+ a_1(f_1, \dots, f_n) T + a_0(f_1, \dots, f_n) = 0,
\end{equation}
for $m \leq M(n, d)$ and $\deg(a_j)\leq D(n, d)$ for $j = 0, \dots, m-1$,
where $M(n, d), D(n, d)\in \NN$ depend only on $n$ and $d$ (and not on $K$).  

We will need the following degree bound on Gr\"obner bases for
arbitrary term orderings.

\begin{theorem}[\cite{Dube}]
Let $K$ be a field and $I\subset K[x_1, \dots, x_n]$ an ideal generated 
by polynomials of degree $\leq d$. Then there exists a Gr\"obner basis $G$
of $I$ with respect to any term order, such that
$$
\deg(f) \leq 2 \left( \frac{d^2}{2} + d\right)^{2^{n-1}}
$$
for $f\in G$.
\end{theorem}

Combined with 
\cite[Proposition 5.15]{atiyah1969introduction}, the existence
of the degree bounded minimal polynomials in \eqref{eq:GBmonpols} now follows from the
lemma below.

\begin{lemma}[\cite{SwSh}]
Suppose that $f, f_1, \dots, f_m\in K[x_1, \dots, x_n]$ and let 
$$
I = (t - f, t_1 - f_1, \dots, t_m - f_m)\subset K[t, t_1, \dots, t_m, x_1, \dots, x_n].
$$
Then 
\begin{enumerate}[(a)]
\item 
$$
  F(f, f_1, \dots, f_m) = 0\qquad\iff\qquad F\in I\cap K[t, t_1, \dots, t_m]
  $$
  for $F\in K[t, t_1, \dots, t_m]$.
\item
Let 
$G$ be a Gr\"obner basis of $I$
with the lexicographic order $x_1 > \cdots > x_n > t > t_1 > \cdots > t_m$.
If $f$ is algebraic over $K(f_1, \dots, f_m)$, then the minimal 
polynomial of $f$ over $K(f_1, \dots, f_m)$ can be identified with
the polynomial in 
$$
G \cap \left(K[t, t_1, \dots, t_m]\setminus K[t_1, \dots, t_m]\right)
$$
with smallest leading term.
\end{enumerate}
\end{lemma}
\begin{proof}
  For $F\in K[t, t_1, \dots, t_m]$,
  $$
  F(f, f_1, \dots, f_m) = 0\qquad\iff\qquad F\in I\cap K[t, t_1, \dots, t_m]
  $$
  is a consequence of the division algorithm. The (elimination) ideal
  $I\cap K[t, t_1, \dots, t_m]$
  is generated by $G\cap K[t, t_1, \dots, t_m]$. The result
  can now be deduced from the fact that the minimal polynomial (after clearing denominators)
  can be identified with the polynomial containing $t$ in $G$ with smallest leading term.
\end{proof}

\newcommand{\germ}{\mathfrak}

\providecommand{\bysame}{\leavevmode\hbox to3em{\hrulefill}\thinspace}
\providecommand{\MR}{\relax\ifhmode\unskip\space\fi MR }
\providecommand{\MRhref}[2]{%
  \href{http://www.ams.org/mathscinet-getitem?mr=#1}{#2}
}
\providecommand{\href}[2]{#2}

\bibliographystyle{amsplain}
\bibliography{main}

\end{document}